\documentclass[reqno,12pt]{amsart}%
\usepackage{amsfonts}
\usepackage{amsmath}
\usepackage{amssymb}
\usepackage{amsaddr}

\newtheorem{theorem}{Theorem}
\theoremstyle{plain}
\newtheorem{corollary}[theorem]{Corollary}
\newtheorem{definition}[theorem]{Definition}

\newtheorem{lemma}[theorem]{Lemma}
\newtheorem{remark}[theorem]{Remark}

\newcommand\bel[1]{\begin{equation}\label{#1}}
\newcommand\ee{\end{equation}}

\numberwithin{theorem}{section}
\numberwithin{equation}{section}

\begin{document}
\title[Time-homogeneous PAM]
{Stochastic Parabolic Anderson Model with Time-homogeneous Generalized Potential:
Mild Formulation of Solution}
\author{Hyun-Jung Kim}
\address{\textnormal{Department of Mathematics\\ 
University of Southern California\\
Los Angeles, CA 90089\\
kim701@usc.edu}}
\curraddr[Hyun-Jung Kim]{Department of Mathematics, University of Southern California\\
Los Angeles, CA 90089}
\urladdr{https://hyunjungkim.org}



\begin{abstract}
A mild formulation for stochastic parabolic Anderson model with time-homogeneous
Gaussian potential suggests a way of defining a solution
to obtain its optimal regularity.
Two different interpretations in the equation or in the mild formulation are possible with
usual pathwise product and the Wick product: the usual pathwise
interpretation is mainly discussed. We emphasize that 
a modified version of parabolic Schauder estimates
is a key idea for the existence and uniqueness of a mild solution. 
In particular, the mild formulation 
is crucial to investigate a relation between the equation 
with usual pathwise product and the Wick product.
\end{abstract}
\maketitle

\section{Introduction}
\label{sec:Intro}

We start with the stochastic parabolic Anderson model (PAM)
\bel{eq000}
\frac{\partial u(t,x)}{\partial t} =
\Delta u(t,x) +
u(t,x) \dot{W}(x)
\ee
driven by multiplicative Gaussian white noise potential in $d$ space dimensions,
where 
$\Delta$ is the Laplacian operator, and 
$\dot{W}(x)$ is time-homogeneous 
Gaussian white noise with mean zero and covariance
$\mathbb{E}\left[\dot W(x)\dot W(y)\right]=\delta(x-y),$
where $\delta$ is the Dirac-delta function.
Note that the white noise is a generalized process, and we need to 
make sense of the multiplication
$u\dot{W}$ in \eqref{eq000}.

Since the notation $\dot{W}(x)$
stands for the formal derivative of a Brownian sheet $W(x)$,
the equation \eqref{eq000} may be written as 
\bel{eq00}
\frac{\partial u(t,x)}{\partial t} =
\Delta u(t,x) +
u(t,x) \frac{\partial^d}{\partial x_1\cdots \partial x_d}W(x).
\ee

Depending on the multiplication between $u$ and
$\displaystyle\frac{\partial^d}{\partial x_1\cdots \partial x_d}W$,
the equation \eqref{eq00} may be classified as follows: 
\begin{itemize}
\item Stochastic PAM with usual pathwise product (Stratonovich interpretation):
\bel{eq00-sp1}
\frac{\partial u(t,x)}{\partial t} =
\Delta u(t,x) +
u(t,x)\cdot \frac{\partial^d}{\partial x_1\cdots \partial x_d}W(x).
\ee
\item Stochastic PAM with Wick product $\diamond$ 
(Wick-It\^o-Skorokhod interpretation):
\bel{eq00-sp1-wick}
\frac{\partial u(t,x)}{\partial t} =
\Delta u(t,x) +
u(t,x)\diamond \frac{\partial^d}{\partial x_1\cdots \partial x_d}W(x).
\ee
\end{itemize}
As the idea of Stratonovich integral suggests,
the equation \eqref{eq00-sp1} 
is equivalently interpreted as a pathwise limit of approximated 
equations
\bel{eq00-sp-approx}
\frac{\partial u^{\varepsilon}(t,x)}{\partial t} =
\Delta u^{\varepsilon}(t,x) +
u^{\varepsilon}(t,x)\cdot 
\frac{\partial^d}{\partial x_1\cdots \partial x_d}W^{\varepsilon}(x),
\ee
where 
$W^{\varepsilon}$ are smooth approximations of each sample 
path of $W$ for $\varepsilon>0$.

If $d\geq 2$, the usual pathwise product of $u \cdot 
\displaystyle\frac{\partial^d}{\partial x_1\cdots \partial x_d}W$
is classically not well-defined, and 
constructing a solution of \eqref{eq00-sp1} 
presents a major challenge in the standard theory
of stochastic partial differential equations.
Indeed, it is known that $d$-dimensional Brownian sheet $W$
has regularity $1/2-\varepsilon$ in each parameter for any $\varepsilon>0$. Therefore,
$\displaystyle\frac{\partial^d}{\partial x_1\cdots \partial x_d}W$ can be
understood to have regularity $-d/2-\varepsilon$ in total.
Then, the solution $u$ is expected to have regularity $2-d/2-\varepsilon$
so that the sum of the regularity of $u$ and 
$\displaystyle\frac{\partial^d}{\partial x_1\cdots \partial x_d}W$ 
is strictly less than 0. Hence, unfortunately,
classical integration theory cannot be applied to
the product $u \cdot  
\displaystyle\frac{\partial^d}{\partial x_1\cdots \partial x_d}W$, and
advanced techniques are inevitably required.

When $d=2$ and on the whole space $\mathbb{R}^2$, 
the paper \cite{Hai15} introduces a particular renormalization procedure and
constructs a solution to \eqref{eq00-sp1} 
by subtracting a divergent constant from the equation.
On a torus of $\mathbb{R}^2$, a solution of \eqref{eq00-sp1} 
is constructed independently
using paracontrolled distributions in \cite{GIP12} and using the theory of regularity
structures in \cite{Hai14}. 
When $d=3$, 
using the theory of regularity structures, the paper \cite{Hai15-2}
carries out the construction of \eqref{eq00-sp1} on the whole space $\mathbb{R}^3$.
An alternative construction of solution to \eqref{eq00-sp1} on a torus
of $\mathbb{R}^3$ is also established in \cite{HP14}.

It turns out that for the model \eqref{eq00-sp1} when $d=1$,
there are several ways to define a solution 
in the Stratonovich sense, and the regularity of solution is worth attention.
In \cite{Hu15}, 
the Feynman-Kac solution for \eqref{eq00-sp1} in the Stratonovich sense 
is introduced. The paper proves that the Feynman-Kac solution 
is almost H\"older 3/4 continuous in time and almost H\"older 1/2
continuous in space. 
Here, ``almost" H\"older continuity of order $\gamma$
means H\"older continuity of any order less than $\gamma$. 
However, the standard parabolic theory implies that
the spatial regularity can be improved.
Indeed, consider the additive model
\begin{equation}
\label{additive}
u_t(t,x)
= u_{xx}(t,x) + \dot{W}(x)
\end{equation}
as a reference for optimal regularity.
It is known that the explicit solution of \eqref{additive} 
is almost H\"older 3/4 continuous in time
and almost H\"older 3/2 continuous in space. 
In that sense, very recently, the paper \cite{KL18}
shows that a solution defined using change of variables is almost 
H\"older 3/4 continuous in time and almost H\"older 3/2 continuous in space
as desired.

On the other hand, the Wick-It\^o-Skorokhod interpretation 
\eqref{eq00-sp1-wick}
draws attention from several authors. 
For example, the paper \cite{UH96} by Uemura when $d=1$,
and the paper \cite{Hu02} by Hu when $d<4$, define the chaos 
solution using multiple It\^o-Wiener integrals. Also, the paper \cite{UH96}
gives some regularity results: 
The chaos solution is almost H\"older 1/2 continuous in both time and space.
Recently, the paper \cite{KL17} pays attention to
the Wick-It\^o-Skorokhod interpretation of \eqref{eq00-sp1-wick} in one space dimension
and
the optimal space-time 
regularity of solution. Using the chaos expansion (or Fourier expansion),
the paper proves that 
the chaos solution is
almost H\"older 3/4 continuous in time and almost H\"older 3/2 continuous in space.

The objectives of this paper are:
\begin{itemize}
\item to introduce a new pathwise solution 
using the mild formulation (called a mild solution) of time-homogeneous 
parabolic Anderson model driven by Gaussian white noise
on an interval with Dirichlet boundary condition
\begin{equation}
\label{eq:main}
\begin{split}
\frac{\partial u(t,x)}{\partial t} &=
\frac{\partial^2 u(t,x)}{\partial x^2} +
u(t,x)\cdot \frac{\partial}{\partial x}W(x),\ t>0, \ 0<x<\pi,\\
u(t,0)&=u(t,\pi)=0, \ u(0,x)=u_0(x).
\end{split}
\end{equation}
\item to obtain the optimal space-time H\"older regularity of the mild solution. That is,
the mild solution is almost H\"older $3/4$ continuous in time and
almost H\"older $3/2$ continuous in space.
\end{itemize}
The main results are extended to any H\"older continuous function 
$W$ on $[0,\pi]$
of order $\gamma \in (0,1)$, which implies
the derivative of $W$ is a generalized function. 
Note that our results can be applied to \eqref{eq:main} on the whole line $\mathbb{R}$ 
as long as an appropriate norm of $W$ on $\mathbb{R}$ is bounded.
We also show the importance of mild formulation:
the mild formulation suggests a way of finding relations 
between Stratonovich interpretation and Wick-It\^o-Skorokhod interpretation.

Section 2 discusses 
the classical Schauder estimate and a modified version of parabolic Schauder estimates.
In Section 3, a mild solution of \eqref{eq:main} is defined and
the optimal space-time regularity of the mild solution is obtained. Section 4 gives conclusion and 
suggests further directions
of research.

\section{The Schauder Estimates}
\label{sec:SE}

In this section, we establish a modified version of the Schauder estimates for parabolic type
on an interval in some spaces.

We start with the H\"older spaces on $\mathbb{R}^d$. Denote
by $\displaystyle\frac{\partial}{\partial z_i}$ the differentiation operator with respect to $z_i$, and
for a multi-index $\alpha = (\alpha_1,\cdots,\alpha_d)$
with $\alpha_i \in \mathbb{N}_0$ and $|\alpha|=\sum_{i=1}^d \alpha_i<\infty$, denote
$$\partial_z^{\alpha} = \frac{\partial^{\alpha_1}}{\partial z_1^{\alpha_1}}\cdots \frac{\partial^{\alpha_d}}{\partial z_d^{\alpha^d}}.$$

Let $G$ be a domain in $\mathbb{R}^d$.
For $0<\gamma < 1$, let
\begin{eqnarray*}
[u]_{\gamma}:= \sup_{z \neq y \in G} 
\frac{\left|u(z)-u(y)\right|}{|z-y|^{\gamma}}.
\end{eqnarray*}
We say that $u$ is H\"older continuous with H\"older exponent $\gamma$
(or H\"older $\gamma$ continuous) on $G$ if
$$
\sup_{z \in G} |u(z)|
+ [u]_{\gamma}<\infty.$$
The collection of H\"older $\gamma$ continuous functions on $G$ is 
denoted by
$\mathcal{C}^{\gamma}\left(G\right)$
with the norm 
$$\| u \|_{\gamma}:=
\|u\|_{L_{\infty}(G)}
+ [u]_{\gamma},$$
where $$\|u\|_{L_{\infty}(G)}:= \sup_{z\in G}|u(z)|.$$

We say that $u$ is a $k$ times continuously differentiable function on $G$
if $\partial^{\alpha}_z u$ exists and is continuous for all $|\alpha|\leq k$.
The collection of $k$ times continuously differentiable functions on $G$ such that
$$\partial^{\alpha}_z u \in \mathcal{C}^{\gamma}(G),\ |\alpha|=k$$
is denoted by $\mathcal{C}^{k+\gamma}(G)$
with the norm
$$\| u \|_{k+\gamma}:= 
\sum_{|\alpha|\leq k} \sup_{z\in G}|\partial^{\alpha}_z u(z)|
+\sum_{|\alpha|=k} [\partial^{\alpha}_z u]_{\gamma}<\infty.$$

In particular, due to the presence of time and space variables, we often write
$$\mathcal{C}^{k_1+\gamma_1,k_2+\gamma_2}_{t,x}\left((0,T)\times G\right)$$
with the norm 
$$\| u \|_{k_1+\gamma_1,k_2+\gamma_2}:= 
\sum_{n\leq k_1|\alpha|\leq k_2} \sup_{t\in (0,T), x\in G}
\left|\partial^n_t\partial^{\alpha}_x u(t,x)\right|
+\left[\partial^{k_1}_t u\right]_{\gamma_1}
+\sum_{|\alpha|=k_2} \left[\partial^{\alpha}_x u\right]_{\gamma_2}
<\infty.$$

We restrict the space domain by $G=(0,\pi)$
and let us denote the Dirichlet heat kernel on $(0,\pi)$ by
$$
P^D(t,x,y)=\sum_{k=1}^{\infty}
e^{-k^2t}\ m_k(x)m_k(y),\ m_k(x)=\sqrt{\frac{2}{\pi}}\sin(kx),\ k\geq 1.
$$
Define a convolution $\star$ for a function $f$ by
$$ \left(P^D \star f\right)(t,x) = \int_0^t\int_0^{\pi} P^D(t-s,x,y)f(s,y)dyds.$$
Let $0<\gamma \notin \mathbb{N}$ and $T>0$ be given.  
The classical parabolic type of Schauder's estimate in H\"older spaces
(Theorem 5.2 of Chapter IV in \cite{Lady67})
says that
the convolution mapping 
$$f \mapsto P^D \star f$$
is continuous from 
$\mathcal{C}_{t,x}^
{\gamma/2,\gamma}\left((0,T)\times(0,\pi)\right)$
to 
$\mathcal{C}_{t,x}^
{1+\gamma/2,2+\gamma}\left((0,T)\times(0,\pi)\right)$.
More specifically, 
there exists a Schauder constant $C_T>0$ 
such that $C_T$ remains bounded as $T\rightarrow 0$ and 
$$\left\|P^D \star f\right\|_{1+\gamma/2,2+\gamma} 
\leq C_T \|f\|_{\gamma/2,\gamma}.$$
Unfortunately, the classical Schauder constant $C_T$ does not give a good estimate for our purpose.
We now give a relaxed version of Schauder's estimate in fractional Sobolev spaces instead of in
H\"older spaces. The modified results will be
useful for the existence of mild solution in the next section.

Denote by $H^s_p$ the fractional Sobolev space 
as the collection of 
all functions $f$ such that
$$\left\|\left(-\Delta\right)^{s/2}f \right\|_{L_p(0,\pi)}<\infty,$$
where $$\left(-\Delta\right)^{s/2}f =\frac{1}{\Gamma(-s/2)}\int_0^{\infty}
(e^{t\Delta}-I)f\frac{dt}{t^{1+\frac{s}{2}}},$$ 
$\Gamma$ is the gamma function, and
$e^{t\Delta}$ is the semigroup of Laplacian operator on $(0,\pi)$ with zero boundary conditions.

Define $$\Lambda:=\left(-\Delta\right)^{1/2}.$$ 
\begin{lemma}
\label{Lem-Kry}
Let $h\in L_p(0,\pi)$, $p\geq 1$ and let $0<\theta\leq 2$. 
For any $0<t\leq T$, we have
$$\left\|\Lambda^{\theta}\int_0^{\pi} P^D(t,\cdot,y)h(y)dy\right\|_{L_{p}(0,\pi)}
\leq C(T,\theta,p)t^{-\theta/2}\|h\|_{L_{p}(0,\pi)}$$
for some $C(T,\theta,p)>0$ depending on $T$, $\theta$ and $p$.
\end{lemma}
\begin{proof}
From
\cite[Lemma 7.3]{Krylov}, it is enough to show that
\begin{equation}
\label{Conti-derivative}
\left\|\partial_t \int_0^{\pi} P^D(t,\cdot,y)h(y)dy\right\|_{L_{p}(0,\pi)}
\leq \frac{C}{t}\|h\|_{L_p(0,\pi)}
\end{equation}
for some $C>0$ depending only on $p$.
Indeed, \eqref{Conti-derivative} follows from the fact
\begin{eqnarray*}
\begin{split}
\left\|\partial_t e^{t\Delta}h\right\|_{L_{p}(0,\pi)}&=
\left\|\partial_t \int_0^{\pi} P^D(t,\cdot,y)h(y)dy\right\|_{L_{p}(0,\pi)}\\
&=\left\|\int_0^{\pi} \partial_t P^D(t,\cdot,y)h(y)dy\right\|_{L_{p}(0,\pi)}
\end{split}
\end{eqnarray*}
and \cite[Lemma 2.5]{CD}.
\end{proof}

\begin{theorem} 
\label{SchauderEstimate}
Let $0<\beta<\gamma<1$ and $T>0$ be given. For $p\geq 1$,
the convolution mapping $f \mapsto P^D \star f$
is continuous from $L_{\infty}\left(0,T;H_p^{\gamma}(0,\pi)\right)$
to $L_{\infty}\left(0,T;H_p^{2+\beta}(0,\pi)\right)$. In particular,
\begin{eqnarray*}
\|P^D \star f\|_{L_{\infty}\left(0,T;H_p^{2+\beta}(0,\pi)\right)} \leq 
CT^{(\gamma-\beta)/2} \|f\|_{L_{\infty}\left(0,T;H_p^{\gamma}(0,\pi)\right)}
\end{eqnarray*}
for some constant $C>0$ depending only on $\beta$ and $\gamma$.
\end{theorem}
\begin{proof}
Let $0<\beta<\gamma<1$ be given.
By the triangle inequality, we have
\begin{eqnarray*}
\left\|P^D \star f (t,\cdot)\right\|_{H_p^{2+\beta}(0,\pi)}
&=&\left\|\int_0^t\int_0^{\pi} P^D(t-s,\cdot,y)f(s,y)dyds\right\|_{H_p^{2+\beta}(0,\pi)}\\
&\leq&\int_0^t \left\| \int_0^{\pi} P^D(t-s,\cdot,y)f(s,y)dy \right\|_{H_p^{2+\beta}(0,\pi)}ds.
\end{eqnarray*}
For each $0<s<t\leq T$ and $0<\theta<2$, there exists a constant $C_0:=C(\theta,T,p)>0$ such that
\begin{equation}
\label{SobolevIneq}
\left\|\Lambda^{\theta}\int_0^{\pi} P^D(t-s,\cdot,y)f(s,y)dy\right\|_{L_{p}(0,\pi)}
\leq C_0(t-s)^{-\theta/2}\|f(s,\cdot)\|_{L_{p}(0,\pi)}
\end{equation}
by Lemma \ref{Lem-Kry}.
Therefore, by the commutativity of $\Lambda$,
\begin{eqnarray*}
\left\|P^D \star f (t,\cdot)\right\|_{H_p^{2+\beta}(0,\pi)}
&\leq& C_0\int_0^t (t-s)^{-(2+\beta-\gamma)/2}\|f(s,\cdot)\|_{H_p^{\gamma}(0,\pi)}ds\\
&\leq& C_0\|f\|_{L_{\infty}(0,T;H_p^{\gamma}(0,\pi))}
\int_0^t (t-s)^{-(2+\beta-\gamma)/2}ds\\
&\leq& C_1\|f\|_{L_{\infty}(0,T;H_p^{\gamma}(0,\pi))}
T^{(\gamma-\beta)/2},
\end{eqnarray*}
completing the proof.
\end{proof}

Denote the Neumann heat kernel on $(0,\pi)$ by
$$
P^N(t,x,y)=\sum_{k=0}^{\infty}
e^{-k^2t}\ \tilde{m}_k(x)\tilde{m}_k(y),\
\tilde{m}_0(x)=\frac{1}{\sqrt{\pi}},\
\tilde{m}_k(x)=\sqrt{\frac{2}{\pi}}\cos(kx),\ k\geq 1.
$$
\begin{corollary}
\label{NeumanSE}
Under the same assumptions of Theorem \ref{SchauderEstimate},
the convolution map
$$f \mapsto P^N \star f$$
is continuous from $L_{\infty}\left(0,T;H_p^{\gamma}(0,\pi)\right)$
to $L_{\infty}\left(0,T;H_p^{2+\beta}(0,\pi)\right)$.
In other words,
\begin{eqnarray*}
\left\|P^N \star f\right\|_{L_{\infty}\left(0,T;H_p^{2+\beta}(0,\pi)\right)} \leq 
CT^{(\gamma-\beta)/2} \|f\|_{L_{\infty}\left(0,T;H_p^{\gamma}(0,\pi)\right)}
\end{eqnarray*}
for some constant $C>0$ depending only on $\beta$ and $\gamma$.
\end{corollary}

\section{The Mild Solution and its Regularity}
\label{sec:MS}

Let $\left\{W(x)\right\}_{x\in [0,\pi]}$ be
a standard Brownian motion
on a probability space $(\Omega,\mathcal{F},\mathbb{P})$.
Here, $\mathcal{F}$ denotes the filtration generated by $W$.
It is known that Brownian motion is in $\mathcal{C}^{\gamma}(0,\pi)$ for any $0<\gamma<1/2$.
Let $W^{\varepsilon}$ be smooth
approximations of $W$ such that
$$\|W^{\varepsilon}-W\|_{\mathcal{C}^{\gamma}(0,\pi)}\rightarrow 0 \ \mbox{as} \
\varepsilon \rightarrow 0.$$

Consider the approximated equations of \eqref{eq:main} for $\varepsilon>0$:
\begin{equation}
\label{eq:mild-approx}
\begin{split}
\frac{\partial u^{\varepsilon}(t,x)}{\partial t} &=
\frac{\partial^2 u^{\varepsilon}(t,x)}{\partial x^2} +
u^{\varepsilon}(t,x)\cdot \frac{\partial}{\partial x}W^{\varepsilon}(x),\ t>0, \ 0<x<\pi,\\
u^{\varepsilon}(t,0)&=u^{\varepsilon}(t,\pi)=0, \ 
u^{\varepsilon}(0,x)=u_0(x).
\end{split}
\end{equation}
Since $W^{\varepsilon}$ is smooth, 
the equation \eqref{eq:mild-approx} has the classical solution $u^{\varepsilon}$.
Denote 
\begin{eqnarray*}
\mathbf{P}_0(t,x)=
\left\{
\begin{array}{ll}
&u_0(x)\ \mbox{if } t=0\\
&\displaystyle\int_0^{\pi} P^D(t,x,y)u_0(y)dy\ \mbox{if } t>0.
\end{array}
\right. 
\end{eqnarray*}
Note that as $t \rightarrow 0^+$, we see that 
$\mathbf{P}_0(t,x)\rightarrow u_0(x)$ for each $x\in[0,\pi]$.
Then, the mild formulation for the equation \eqref{eq:mild-approx}
is given by
\begin{equation}
\label{PAM-mild-approx}
u^{\varepsilon}(t,x)=
\mathbf{P}_0(t,x)+
\int_0^t\int_0^{\pi} P^D(t-s,x,y)u^{\varepsilon}(s,y)\partial_yW^{\varepsilon}(y)dyds.
\end{equation}
By integration by parts with respect to $y$ in the last term above,
we rewrite \eqref{PAM-mild-approx} as
\begin{eqnarray*}
\begin{split}
u^{\varepsilon}(t,x)& = \mathbf{P}_0(t,x)-
\int_0^t\int_0^{\pi} P^D_y(t-s,x,y)u^{\varepsilon}(s,y)W^{\varepsilon}(y)dyds\\
&- \int_0^t\int_0^{\pi} P^D(t-s,x,y)u^{\varepsilon}_y(s,y)W^{\varepsilon}(y)dyds.
\end{split}
\end{eqnarray*}

\begin{definition}
We say that $u$ is a mild solution of \eqref{eq:main} if
\begin{itemize}
\item for every $0<t<T$ and $0<x<\pi$,
$u$ is continuous in $t$,
continuously differentiable in $x$, and
it satisfies the equation
\begin{equation}
\label{PAM-mild}
\begin{split}
u(t,x)&=\mathbf{P}_0(t,x)-
\int_0^t\int_0^{\pi} P^D_y(t-s,x,y)u(s,y)W(y)dyds\\
&- \int_0^t\int_0^{\pi} P^D(t-s,x,y)u_y(s,y)W(y)dyds
\end{split}
\end{equation}
\item for every $0\leq t\leq T$, $u(t,0)=u(t,\pi)=0$;
\item for every $0\leq x\leq \pi$, $\lim_{t\rightarrow 0^+} u(t,x)=u_0(x)$.
\end{itemize}
\end{definition}

For the existence and the uniqueness of mild solution,
we use a contraction mapping (or fixed point argument) on 
$L_{\infty}\left(0,T;H_p^{1+\beta}(0,\pi)\right)$ with $0<\beta<1$ and $p \geq 1$.
Define a map 
$$ \mathcal{M}: L_{\infty}\left(0,T;H_p^{1+\beta}(0,\pi)\right)
\rightarrow L_{\infty}\left(0,T;H_p^{1+\beta}(0,\pi)\right)$$
by
\begin{equation}
\label{PAM-mild-int-map}
\begin{split}
\left(\mathcal{M}u\right)(t,x)&=\mathbf{P}_0(t,x)-
\int_0^t\int_0^{\pi} P^D_y(t-s,x,y)u(s,y)W(y)dyds\\
&-\int_0^t\int_0^{\pi} P^D(t-s,x,y)u_y(s,y)W(y)dyds.
\end{split}
\end{equation}

We first prove the well-posedness of \eqref{PAM-mild-int-map}.
\begin{lemma}
\label{initialregB}
Let $0<\beta<\gamma<1$.
If $u_0 \in H_p^{\beta}(0,\pi)$ with $p \geq 1$,
$$\left\|\mathbf{P}_0\right\|_
{L_{\infty}(0,T;H_p^{1+\gamma}(0,\pi))}
<\infty.$$
\end{lemma}
\begin{proof}
It is clear from Lemma \ref{Lem-Kry}.
\end{proof}

\begin{theorem}
\label{fixedpt}
Let $0<\beta<\gamma<1/2$.
If $u_0 \in H_p^{\beta}(0,\pi)$ with $p \geq 1$,
the mapping $\mathcal{M}$ 
on $L_{\infty}\left(0,T;H_p^{1+\beta}(0,\pi)\right)$
is well-defined.
Also, there exists a fixed point of $\mathcal{M}$. 
That is, the fixed point is
the unique mild solution of \eqref{eq:main}.
\end{theorem}
\begin{proof}
Since $P^D_y(t,x,y)=-P^N_x(t,x,y)$ for each $t>0$ and $x,y\in [0,\pi]$,
we have $$\int_0^t\int_0^{\pi}  P^D_y(t-s,x,y)u(s,y)W(y)dyds
= -\int_0^t\int_0^{\pi} P^N_x(t-s,x,y)u(s,y)W(y)dyds.$$
Then, we rewrite \eqref{PAM-mild-int-map} by
\begin{eqnarray*}
\left(\mathcal{M}u\right)(t,x)= \mathbf{P}_0(t,x)+
P^N_x\star(uW)(t,x)
- P^D\star(u_xW)(t,x).
\end{eqnarray*}

We show the well-definedness term by term.
By Lemma \ref{initialregB},
$$\mathbf{P}_0
\in L_{\infty}(0,T;H_p^{1+\gamma}(0,\pi)).$$

Fix $t \in (0,T)$.
We have
\begin{eqnarray*}
\begin{split}
\left\|P^N_x\star uW(t,\cdot)\right\|_{H_p^{1+\beta}(0,\pi)}=&
\left\|\partial_x\left(P^N\star uW\right)(t,\cdot)\right\|_{H_p^{1+\beta}(0,\pi)}\\
\leq&
\left\|P^N \star (uW)(t,\cdot)\right\|_{H_p^{2+\beta}(0,\pi)}.
\end{split}
\end{eqnarray*}
Then, by Corollary \ref{NeumanSE} and \cite[Lemma 5.2]{Krylov},
\begin{eqnarray*}
\begin{split}
\left\|P^N_x\star uW\right\|_{L_{\infty}(0,T;H_p^{1+\beta}(0,\pi))}
\leq& \left\|P^N\star uW\right\|_{L_{\infty}(0,T;H_p^{2+\beta}(0,\pi))}\\
\leq& C_1T^{(\gamma-\beta)/2}\|W\|_{\mathcal{C}^{\gamma}(0,\pi)}
\|u\|_{L_{\infty}\left(0,T;H_p^{\gamma}(0,\pi)\right)}\\
\leq& C_2T^{(\gamma-\beta)/2}\|W\|_{\mathcal{C}^{\gamma}(0,\pi)}
\|u\|_{L_{\infty}\left(0,T;H_p^{1+\beta}(0,\pi)\right)}
\end{split}
\end{eqnarray*}
and similarly by \eqref{SobolevIneq},
\begin{eqnarray*}
\begin{split}
\|P^D\star u_xW\|_{L_{\infty}(0,T;H_p^{1+\beta}(0,\pi))}
\leq C_3T^{1/2}\|W\|_{\mathcal{C}^{\gamma}(0,\pi)}
\|u\|_{L_{\infty}\left(0,T;H_p^{1+\beta}(0,\pi)\right)}.
\end{split}
\end{eqnarray*}

Let $u,v\in L_{\infty}\left(0,T;H_p^{1+\beta}(0,\pi)\right)$. Then,
\begin{eqnarray*}
\begin{split}
\|\left(\mathcal{M}u-\mathcal{M}v\right)\|
_{L_{\infty}\left(0,T;H_p^{1+\beta}(0,\pi)\right)}\leq&
\left\|P^N_x\star (u-v)W\right\|_{L_{\infty}\left(0,T;H_p^{1+\beta}(0,\pi)\right)}\\
+&
\|P^D \star (u_x-v_x)W\|_{L_{\infty}\left(0,T;H_p^{1+\beta}(0,\pi)\right)}\\
\leq& CT^{1/2}\|W\|_{\mathcal{C}^{\gamma}(0,\pi)}\|u-v\|_{L_{\infty}\left(0,T;H_p^{1+\beta}(0,\pi)\right)}
\end{split}
\end{eqnarray*}
for some $C>0$.
Choose $\delta>0$ such that $C\delta^{1/2}\|W\|_{\mathcal{C}^{\gamma}(0,\pi)}<1$. 
Then, clearly there exists a fixed point of $\mathcal{M}$ up to $\delta$.
Consider a time partition $0=t_0<\cdots<t_n=T$ such that
$t_{i+1}-t_i\leq \delta$ for $i=0,\cdots, n-1$. 
We now define $u$ recursively 
on $(t_i,t_{i+1}]$ with the initial condition $u(t_i,x)$ for $i=0,\cdots,n-1$:
\begin{equation}
\label{Global-sol}
\begin{split}
(\mathcal{M}u)(t,x)=&\int_0^{\pi} P^D(t_{i+1},x,y)u(t_i,y)dy \\
-&\int_0^{t_{i+1}}\int_0^{\pi} P^D_y(t_{i+1}-s,x,y)u(s,y)W(y)dyds\\
-&\int_0^{t_{i+1}}\int_0^{\pi} P^D(t_{i+1}-s,x,y)u_y(s,y)W(y)dyds.
\end{split}
\end{equation}
The existence of the fixed point to \eqref{Global-sol} 
is guaranteed by the above arguments since
$t_{i+1}-t_i\leq \delta$. Since $n$ is finite,
we obtain the fixed point solution over the whole time interval $(0,T)$.
We note that the fixed point solution satisfies the mild formulation \eqref{PAM-mild}.
Since the fixed point solution is unique, the uniqueness of mild solution 
clearly holds.
\end{proof}

\begin{remark}
\label{SobolevEmb}
From Theorem \ref{fixedpt}, we have that for all $t\in [0,T]$,
$$u(t,\cdot)\in H_p^{1+\beta}(0,\pi),\ p \geq 1.$$
By the Sobolev embedding theorem, we have
\begin{eqnarray*}
H_p^{1+\beta}(0,\pi) \subset \mathcal{C}^{1+\beta-1/p}(0,\pi)
\end{eqnarray*}
for any $p \geq 1$. 
This shows that the mild solution $u$ is
indeed almost H\"older 3/2 continuous in space.
\end{remark}

We show the mild solution of \eqref{eq:main} is 
the limit of classical solutions $u^{\varepsilon}$ 
of \eqref{eq:mild-approx} in $L_{\infty}(0,T;H_p^{1+\beta}(0,\pi))$.
\begin{theorem}
\label{HolderConvergence}
Let $0<\beta<\gamma<1/2$.
If $u_0\in H_p^{\beta}(0,\pi)$,
then we have
$$\|u^{\varepsilon}-u\|_{L_{\infty}(0,T;H_p^{1+\beta}(0,\pi))}
\rightarrow 0\ \mbox{as}
\ \varepsilon \rightarrow 0.$$
\end{theorem}
\begin{proof}
For simplicity,
we denote $\|\cdot\|_{\mathcal{C}^{\gamma}}
=\|\cdot\|_{\mathcal{C}^{\gamma}(0,\pi)}$ and
$$\|\cdot\|=\|\cdot\|_{L_{\infty}\left(0,T;H_p^{1+\beta}(0,\pi)\right)}.$$
Then, we can write $u^{\varepsilon}-u$ by
\begin{eqnarray*}
P^N_x\star u^{\varepsilon}W^{\varepsilon}-P^N_x\star uW
-P^D\star u_x^{\varepsilon}W^{\varepsilon}
+P^D\star u_xW.
\end{eqnarray*}

By the triangle inequality,
\begin{eqnarray*}
\begin{split}
\|u^{\varepsilon}-u\|&\leq 
\left\|P^N_x\star (u^{\varepsilon}-u)W^{\varepsilon}\right\|
+\left\|P^N_x\star u(W-W^{\varepsilon})\right\|\\
&+\|P^D\star (u_x^{\varepsilon}-u_x)W^{\varepsilon}\|
+\|P^D\star u_x(W-W^{\varepsilon})\|\\
&\leq 
C'_1T^{1/2}\|W^{\varepsilon}\|_{\mathcal{C}^{\gamma}}\|u^{\varepsilon}-u\|
+C'_2T^{1/2}\|u\|\|W-W^{\varepsilon}\|_{\mathcal{C}^{\gamma}}\\
&\leq  
C'_1T^{1/2}\left(\|W\|_{\mathcal{C}^{\gamma}}+\|W-W^{\varepsilon}\|_{\mathcal{C}^{\gamma}}\right)
\|u^{\varepsilon}-u\|
+C'_2T^{1/2}\|u\|\|W-W^{\varepsilon}\|_{\mathcal{C}^{\gamma}}.
\end{split}
\end{eqnarray*}
Note that the constants $C'_1,C'_2>0$ are independent of $\varepsilon$.
Choose $\delta>0$ such that 
$$C'_1\delta^{1/2}\left(\|W\|_{\mathcal{C}^{\gamma}}+
\|W-W^{\varepsilon}\|_{\mathcal{C}^{\gamma}}\right)<1$$
for small $\varepsilon>0$.
Then, since $\|W-W^{\varepsilon}\|_{\mathcal{C}^{\gamma}} \rightarrow 0,$
$$\|u^{\varepsilon}-u\| \rightarrow 0\
\mbox{as} \ \varepsilon \rightarrow 0.$$

We now consider a time partition $0=t_0<\cdots<t_n=T$ such that
$t_{i+1}-t_i\leq \delta$ for $i=0,\cdots, n-1$. 
Finally, we iterate the above argument for $u^{\varepsilon}-u$ recursively 
on $(t_i,t_{i+1}]$ to get
$$\|u^{\varepsilon}(t_i,\cdot)-u(t_i,\cdot)\|\rightarrow 0 \ \mbox{as}
\ \varepsilon \rightarrow 0$$
for $i=1,\cdots,n-1$.
\end{proof}

\begin{theorem}
\label{optimalregularity}
Let $0<\gamma<1/2$. If $u_0\in \mathcal{C}^{1+\gamma}(0,\pi)$,
then the mild solution of \eqref{eq:main} is indeed in 
$$\mathcal{C}^{(1+\gamma)/2,1+\gamma}_{t,x}
\left((0,T)\times (0,\pi)\right).$$
\end{theorem}
\begin{proof}
Let $0<\beta<\gamma<1/2$.
Recall, for each $\varepsilon>0$, the approximated parabolic Anderson model
\begin{eqnarray*}
\begin{split}
\frac{\partial u^{\varepsilon}(t,x)}{\partial t} &=
\frac{\partial^2 u^{\varepsilon}(t,x)}{\partial x^2} +
u^{\varepsilon}(t,x)\cdot \frac{\partial}{\partial x}W^{\varepsilon}(x),\ t>0, \ 0<x<\pi,\\
u^{\varepsilon}(t,0)&=u^{\varepsilon}(t,\pi)=0, \ 
u^{\varepsilon}(0,x)=u_0(x).
\end{split}
\end{eqnarray*}
From the classical parabolic theory,
there exists the unique classical solution $u^{\varepsilon}$.
We note that by Theorem \ref{HolderConvergence},
$u^{\varepsilon}$
converges to the limit $u$ in 
$$L_{\infty}(0,T;H_p^{1+\beta}(0,\pi))$$
since $\mathcal{C}^{1+\gamma}(0,\pi)
\subset H_p^{\beta}(0,\pi)$ for any $p \geq 1.$

On the other hand, let
$v^{\varepsilon}$ satisfy
the equation
\begin{eqnarray*}
\begin{split}
\frac{\partial v^{\varepsilon}(t,x)}{\partial t} &=
\frac{\partial^2 v^{\varepsilon}(t,x)}{\partial x^2} -2
W^{\varepsilon}\cdot \frac{\partial}{\partial x}v^{\varepsilon}(t,x)
+(W^{\varepsilon})^2v^{\varepsilon}(t,x),\ t>0, \ 0<x<\pi,\\
v^{\varepsilon}(t,0)&=v^{\varepsilon}(t,\pi)=0, \ 
v^{\varepsilon}(0,x)=u_0(x)e^{-\int_0^x W^{\varepsilon}(y)dy}.
\end{split}
\end{eqnarray*}
We observe that
$$u^{\varepsilon}(t,x)=v^{\varepsilon}(t,x)e^{\int_0^x W^{\varepsilon}(y)dy}.$$
By \cite[Theorem 2.3]{KL18},
$u^{\varepsilon}$ converges to a limit in 
$$\mathcal{C}^{(1+\gamma)/2,1+\gamma}_{t,x}\left((0,T) \times (0,\pi)\right)$$
and
the limit of $u^{\varepsilon}$ is unique 
in $L_2\left((0,T);H_0^1(0,\pi)\right) \bigcap L_{\infty}\left((0,T);L_2(0,\pi)\right)$
by \cite[Theorem 3.5]{KL18},
where $H_0^1(0,\pi)$ is the closure of the set of smooth functions with compact support in $(0,\pi)$
with respect to the norm $\|\cdot\|_{H^1_2(0,\pi)}$.

Since, for $p \geq 2$, $$L_{\infty}(0,T;H_p^{1+\beta}(0,\pi)) 
\subset L_2\left((0,T);H_0^1(0,\pi)\right) \bigcap L_{\infty}\left((0,T);L_2(0,\pi)\right),$$
the mild solution $u$ of \eqref{eq:main} is indeed
in $$\mathcal{C}^{(1+\gamma)/2,1+\gamma}_{t,x}\left((0,T)\times (0,\pi)\right).$$
\end{proof}

\begin{remark}
The reason why we set the upper bound of regularity $\gamma$ less than $1/2$ 
in Theorem \ref{fixedpt}, \ref{HolderConvergence} and \ref{optimalregularity}
is due to the regularity of Brownian motion.
In fact, the Brownian motion $W$ can be 
replaced by any pathwisely H\"older $\gamma$ continuous process 
with $0<\gamma<1$ for
Theorem \ref{fixedpt}, \ref{HolderConvergence} and \ref{optimalregularity}. 
For example, $W$ can be a standard fractional Brownian motion $W^H$ with 
the Hurst index $0<H<1$.
\end{remark}

\begin{remark}
Consider the following equation on the whole line $\mathbb{R}$
\begin{eqnarray*}
\begin{split}
\frac{\partial u(t,x)}{\partial t} &=
\frac{\partial^2 u(t,x)}{\partial x^2} +
u(t,x)\cdot \frac{\partial}{\partial x}W(x),\ t>0, \ x\in \mathbb{R},\\
u(0,x)&=u_0(x),\ x\in \mathbb{R}.
\end{split}
\end{eqnarray*}
Theorem \ref{fixedpt}, \ref{HolderConvergence} and \ref{optimalregularity}
will also work on the whole line $\mathbb{R}$
if the H\"older $\gamma$ norm of $W$ on $\mathbb{R}$ is bounded
with $0<\gamma<1$;
\begin{itemize}
\item
The modified Schauder estimate result in Theorem \ref{SchauderEstimate} 
is sharper in H\"older spaces 
if $P^D$ is replaced by the Gaussian heat kernel $P(t,x-y)$ on $\mathbb{R}$:
for any $f\in \left(0,T;\mathcal{C}^{\gamma}(\mathbb{R})\right)$, we have
$$\left\|P \star f\right\|_{L_{\infty}
\left(0,T;\mathcal{C}^{2+\gamma}(\mathbb{R})\right)} \leq 
CT^{\gamma/2} \|f\|_{L_{\infty}\left(0,T;\mathcal{C}^{\gamma}(\mathbb{R})\right)},$$
for some $C>0$, which is independent of $T$.
Note that a Brownian motion on $\mathbb{R}$ do not have
a sample trajectory that has a bounded H\"older norm.
\item
Clearly, $P_y(t,x-y)=P_x(t,x-y)$ holds for $t>0$.
\end{itemize}
\end{remark}

\section{Conclusion and Further Directions}
\label{sec:FD}

\subsection{Spatial Optimal Regularity}

The paper \cite{Hu15} gives the spatial
(H\"older) regularity only less than $1/2$. However,
Theorem \ref{fixedpt} using the fixed point argument and
\cite[Theorem 2.3]{KL18}
show that
the optimal spatial (H\"older) regularity of the solution $u$ is $3/2-\varepsilon$
for any $\varepsilon>0$ as long as $u_0 \in \mathcal{C}^{3/2}(0,\pi)$.
The result implies several important remarks:
\begin{itemize}
\item Achieve the spatial regularity higher than $1/2$.
Since the standard Brownian motion $W$ is H\"older $1/2-\varepsilon$ 
continuous almost surely, it is possible to apply Young's integral:
For each $s<t$ and $x$,
$$
\int_0^{\pi} P(t-s,x,y)u(s,y)\frac{\partial}{\partial y}W(y)dy:=
\int_0^{\pi} P(t-s,x,y)u(s,y)dW(y)
$$
appearing in the classical mild formulation of \eqref{eq:main};
\item The regularity $3/4-\varepsilon$ in time and $3/2-\varepsilon$ in space for
$\varepsilon>0$
are indeed in line with the standard parabolic partial differential equation theory.
\end{itemize}

\subsection{Why Mild Formulation?}

In fact, the mild formulation is also applied to construct the Wick-It\^o-Skorokhod solution 
of \eqref{eq00-sp1-wick} on $[0,\pi]$ with Dirichlet boundary condition 
in \cite{KL17}. It is known that for
any initial function $u_0 \in \mathcal{C}^{3/2}(0,\pi)$, there exists the unique
Wick-It\^o-Skorokhod solution $u^{\diamond}$
satisfying the equation
\bel{wickmild}
u^{\diamond}(t,x)=\mathbf{P}_0(t,x)
+\int_0^t\int_0^{\pi} P(t-s,x,y)u^{\diamond}(s,y)\diamond \dot W(y)dyds
\ee
almost surely in $\mathcal{C}_{t,x}^{3/4-\varepsilon, 3/2-\varepsilon}\left((0,T)\times (0,\pi)\right)$
for any $\varepsilon>0$.

The natural further question is to find a
meaningful relation 
between the usual solution of \eqref{eq:main} and 
the Wick-It\^o-Skorokhod solution of \eqref{eq00-sp1-wick}
with Dirichlet boundary condition
in the mild formulation.

There are relations between the usual product and the Wick product
(e.g. \cite{Luo} and \cite{MikRoz12}).
Let $\xi_k$' be i.i.d standard Gaussian random variables. 
Denote by $\mathcal{S}$ the collection of all multi-indices 
$\alpha=(\alpha_1,\alpha_2,\cdots)$ 
such that
$\alpha_k \in \mathbb{N}_0,\ k=1,2,\cdots,$ and
$\displaystyle\sum_{k=1}^{\infty} \alpha_k <\infty.$ For 
$\alpha,\beta \in \mathcal{S}$,
define
\begin{itemize}
\item $(0)=(0,0,\cdots)$;
\item $\epsilon(k)$ is the multi-index $\alpha$ such that 
$\alpha_k=1$ and $\alpha_i=0$ if $i\neq k$;
\item $\alpha+\beta=(\alpha_1+\beta_1,\alpha_2+\beta_2,\cdots)$;
\item $\alpha-\beta=(\max(\alpha_1-\beta_1,0),\max(\alpha_2-\beta_2,0),\cdots)$;
\item $\alpha!=\prod_k \alpha_k!$.
\end{itemize}

We define the Hermite polynomial of order $n$ by
$$\mathbf{H}_n(x)=(-1)^n e^{x^2/2} \displaystyle\frac{d^n}{dx^n}e^{-x^2/2}$$
and define, for each $\alpha \in \mathcal{S}$,
$$\xi_{\alpha}=\prod_{k=1}^{\infty} \left(\frac{\mathbf{H}_{\alpha_k}(\xi_k)}{
\sqrt{\alpha_k!}}\right).$$
For any $u,v \in L_2(\Omega)$,
Cameron-Martin theorem \cite{CM47} gives the following fact:
For each $t\in [0,T]$ and $x\in [0,\pi]$,
$u$ and $v$ 
may be written as
$$u(t,x)=\sum_{\alpha\in \mathcal{S}}u_{\alpha}(t,x)\xi_{\alpha}\ 
\mbox{and} \ v(t,x)=\sum_{\alpha\in \mathcal{S}}v_{\alpha}(t,x)
\xi_{\alpha}.
$$ 
Also, 
we have the identity \cite[Theorem 2.3]{Luo}
$$u\cdot v=u\diamond v +
\sum_{\alpha\in \mathcal{S}}\left(\sum_{\gamma \neq (0)} \sum_{(0)\leq \beta \leq \alpha}
\frac{\sqrt{\alpha!(\alpha-\beta+\gamma)!(\beta+\gamma)!}}{\beta!\gamma!(\alpha-\beta)!}
u_{\alpha-\beta+\gamma}v_{\beta+\gamma}\right)\xi_{\alpha}.
$$

From the fact that
standard Brownian motion on $[0,\pi]$ has an explicit formula 
$$W(x)=\sum_{k=1}^{\infty} \left(\int_0^x m_k(y)dy\right) \xi_{\epsilon(k)},$$
we have a formal expression of Gaussian white noise on $[0,\pi]$ given by
$$\dot W(x)=\sum_{k=1}^{\infty} m_k(x) \xi_{\epsilon(k)},$$ where $m_k$'s
are defined as before.
Consider smooth approximations $\dot W^{\varepsilon}$ of $\dot W(x)$ by a convolution with
$\phi_{\varepsilon}$ as above: For each $x\in [0,\pi]$ and
$1\leq p <\infty$,
$$\dot W^{\varepsilon}(x)=\sum_{k=1}^{\infty} m_k^{\varepsilon}(x)\xi_{\epsilon(k)}
\in L_p(\Omega),$$
where $m_k^{\varepsilon}(x)=m_k \ast \phi_{\varepsilon}(x).$

It is also known \cite{KL17} that for any $u_0\in L_p(0,\pi),\ 1\leq p <\infty$, 
the Wick-It\^o-Skorokhod solution of \eqref{wickmild} has 
the basic regularity
$$u^{\diamond}(t,x)\in L_p(\Omega),\ t>0,\ x\in[0,\pi].$$
Also, the approximated Wick-It\^o-Skorokhod solutions in the mild formulation
$$
\left(u^{\varepsilon}\right)^{\diamond}(t,x)=\mathbf{P}_0(t,x)
+\int_0^t\int_0^{\pi} P(t-s,x,y)
\left(u^{\varepsilon}\right)^{\diamond}(s,y)\diamond \dot W^{\varepsilon}(y)dyds.
$$
have the regularity
$$\left(u^{\varepsilon}\right)^{\diamond}(t,x)\in L_p(\Omega),\ t>0,\ x\in[0,\pi].$$

Then, we get the relation
\bel{relation}
\left(u^{\varepsilon}\right)^{\diamond}(t,x)\cdot \dot W^{\varepsilon}(x) = \left(u^{\varepsilon}\right)^{\diamond}(t,x)\diamond \dot W^{\varepsilon}(x) +
\sum_{\alpha\in \mathcal{S}}\sum_{k\geq 1}\left(\sqrt{\alpha_k+1}
\left(u^{\varepsilon}\right)^{\diamond}_{\alpha+\epsilon(k)}(t,x)
m^{\varepsilon}_{k}(x)\right)\xi_{\alpha}.
\ee

Define the residual by
$$\eta^{\varepsilon}(t,x)=
\displaystyle\sum_{\alpha\in \mathcal{S}} \eta^{\varepsilon}_{\alpha}(t,x)\xi_{\alpha}
:=\sum_{\alpha\in \mathcal{S}}\sum_{k\geq 1}\left(\sqrt{\alpha_k+1}
\left(u^{\varepsilon}\right)^{\diamond}_{\alpha+\epsilon(k)}(t,x)m^{\varepsilon}_{k}(x)\right)\xi_{\alpha},$$
where
$$\eta^{\varepsilon}_{\alpha}(t,x):=\sum_{k\geq 1}\sqrt{\alpha_k+1}
\left(u^{\varepsilon}\right)^{\diamond}_{\alpha+\epsilon(k)}(t,x)m^{\varepsilon}_{k}(x).$$

Beyond the basic relations, let us find a further connection 
between usual solution and Wick-It\^o-Skorokhod solution 
of \eqref{eq00}.
Consider the approximated mild solutions of \eqref{eq:mild-approx}
\bel{mildusual2}
u^{\varepsilon}(t,x)=\mathbf{P}_0(t,x)
+\int_0^t\int_0^{\pi} P(t-s,x,y)u^{\varepsilon}(s,y)\cdot \dot W^{\varepsilon}(y)dyds.
\ee

After we define $Z^{\varepsilon}(t,x)$
by $$Z^{\varepsilon}(t,x)=u^{\varepsilon}(t,x)-\left(u^{\varepsilon}\right)^{\diamond}(t,x)$$
and using the relation \eqref{relation}, we have the equation
\bel{difference}
Z^{\varepsilon}(t,x)=\int_0^t\int_0^{\pi} P(t-s,x,y)Z^{\varepsilon}(s,y)\cdot \dot W^{\varepsilon}(y)dyds
-\int_0^t\int_0^{\pi} P(t-s,x,y)\eta^{\varepsilon}(s,y)dyds.
\ee
Equivalently, the equation \eqref{difference} is the mild formulation of
\bel{eq:difference}
\begin{split}
&\frac{\partial Z^{\varepsilon}(t,x)}{\partial t}
=\frac{\partial^2 Z^{\varepsilon}(t,x)}{\partial x^2}+
Z^{\varepsilon}(t,x)\cdot \dot W^{\varepsilon}(x)-\eta^{\varepsilon}(t,x),
\ 0<t<T,\ 0<x<\pi; \\
&Z^{\varepsilon}(0,x)=0, \ Z^{\varepsilon}(t,0)=Z^{\varepsilon}(t,\pi)=0.
\end{split}
\ee
Set 
$$Z(t,x)=u(t,x)-u^{\diamond}(t,x),$$ where $u$ is the usual mild solution of \eqref{eq:main}
and $u^{\diamond}$ is the Wick-It\^o-Skorokhod solution of \eqref{eq00-sp1-wick}
on $[0,\pi]$
with Dirichlet boundary condition.
Then, for any $0<\gamma_1<3/4$ and $0<\gamma_2<3/2$,
$$Z^{\varepsilon} \rightarrow Z\ \mbox{in} 
\ \mathcal{C}^{\gamma_1,\gamma_2}_{t,x}\left((0,T)\times(0,\pi)\right)
\ \mbox{as} \ \varepsilon \rightarrow 0.$$
We naturally expect that
$Z(t,x)$ satisfies
the equation
\bel{eq:difference-1}
\begin{split}
&\frac{\partial Z(t,x)}{\partial t}
=\frac{\partial^2 Z(t,x)}{\partial x^2}+Z(t,x)\cdot \dot W(x)-\eta(t,x),
\ 0<t<T,\ 0<x<\pi; \\
&Z(0,x)=0, \ Z(t,0)=Z(t,\pi)=0,
\end{split}
\ee
where $$\eta(t,x)=
\sum_{\alpha\in \mathcal{S}}\sum_{k\geq 1}\left(\sqrt{\alpha_k+1}
u^{\diamond}_{\alpha+\epsilon(k)}(t,x)m_{k}(x)\right)\xi_{\alpha}.$$
Actually, we can show,
for each $t>0$, $\eta^{\varepsilon}(t,\cdot)$ converges to $\eta(t,\cdot)$
in $L_2\left(\Omega;\mathcal{H}_2^{-r}(0,\pi)\right)$ with $r>1/2$.
Here, $\mathcal{H}_2^{-r}$ is the dual space of the Sobolev space $\mathcal{H}_2^r$.
Therefore,
we can view the Wick-It\^o-Skorokhod solution of \eqref{eq00-sp1-wick}
on $[0,\pi]$
with Dirichlet boundary condition
as an approximation of usual mild solution of \eqref{eq:main}, and moreover,
further investigation of the residual equation \eqref{eq:difference}
is reasonable to give a rigorous relation between
usual solution and Wick-It\^o-Skorokhod solution.

\def\cprime{$'$}
\providecommand{\bysame}{\leavevmode\hbox to3em{\hrulefill}\thinspace}
\providecommand{\MR}{\relax\ifhmode\unskip\space\fi MR }
\providecommand{\MRhref}[2]{%
  \href{http://www.ams.org/mathscinet-getitem?mr=#1}{#2}
}
\providecommand{\href}[2]{#2}

\end{document}